\documentclass[11pt, a4paper]{amsart}

\usepackage{amsmath}
\usepackage{amssymb}
\usepackage{graphicx}
\usepackage{color}
\usepackage{url}
\usepackage{blkarray}
\usepackage{hyperref}
\hypersetup{
    colorlinks,
    citecolor=black,
    filecolor=black,
    linkcolor=black,
    urlcolor=black
}

\newtheorem{thm}{Theorem}
\newtheorem*{que*}{Question}
\newtheorem{prop}[thm]{Proposition}
\newtheorem{lem}[thm]{Lemma}
\newtheorem{example}[thm]{Example}

\newtheorem{remark}[thm]{Remark}

\begin{document}

\title[Bi-Perron numbers with real or unimodular conjugates]{The geometry of bi-Perron numbers with\\ real 
or unimodular Galois conjugates}

\author{Livio Liechti}
\address{Department of Mathematics\\
University of Fribourg\\
Chemin du Mus\'ee 23\\
1700 Fribourg\\Switzerland}
\email{livio.liechti@unifr.ch}

\author{Joshua Pankau}
\address{Department of Mathematics\\ University of Iowa \\ 14 MacLean Hall\\ Iowa City, Iowa 52242-1419}
\email{joshua-pankau@uiowa.edu}

\begin{abstract} 
Among all bi-Perron numbers, we characterise those all of whose Galois conjugates are 
real or unimodular as the ones that admit a power which is the stretch 
factor of a pseudo-Anosov homeomorphism arising from Thurston's construction.
This is in turn equivalent to admitting a power which is the spectral radius of a bipartite 
Coxeter transformation.  
\end{abstract}

\maketitle

\section{Introduction}

\noindent
Particular geometric situations often give rise to particular algebraic numbers, 
and it is a natural question to characterise these numbers by their geometry.
In this note, we provide a 
description of those bi-Perron numbers all of whose Galois conjugates are real or unimodular. 
We relate those numbers to pseudo-Anosov stretch factors arising via Thurston's construction, 
and to the spectral radii of bipartite Coxeter transformations.  
\medskip 

\noindent
A {bi-Perron} number~$\lambda$ is a real algebraic unit~$>1$ all of whose 
Galois conjugates have modulus in the open interval~$(\lambda^{-1}, \lambda)$, 
except for~$\lambda$ itself and possibly one of~$\pm\lambda^{-1}$.
Stretch factors of pseudo-Anosov homeomorphisms are prominent examples of bi-Perron numbers. 
Conversely, the following problem was posed by Fried~\cite{Fried85}: does every bi-Perron number have 
a power\footnote{Fried's problem is sometimes cited in a stronger form that does not allow powers. 
While powers are needed in the setting of our Theorem~\ref{bi-Perron_characterisation}, 
indeed they might not be necessary if we allow all pseudo-Anosov homeomorphisms instead of just the ones arising from Thurston's construction.
However, even the version with powers suffices to ensure that every bi-Perron number arises as the 
growth rate of a surface homeomorphism.} that arises as the stretch factor of a pseudo-Anosov homeomorphism?
Recently, the second author gave a positive answer to Fried's problem for
the class of Salem numbers~\cite{Pankau}, 
that is, bi-Perron numbers with all other Galois conjugates of modulus at most one, 
and with at least one Galois conjugate of modulus exactly one. Our main result extends this positive answer to the class of 
all bi-Perron numbers with real or unimodular Galois conjugates. This condition on the Galois conjugates 
turns out to precisely characterise the pseudo-Anosov stretch factors that arise from Thurston's construction, 
and the spectral radii of bipartite Coxeter transformations.

\begin{thm}
\label{bi-Perron_characterisation}
For a bi-Perron number~$\lambda$, the following are equivalent.
\begin{enumerate}
\item[(a)] All Galois conjugates of~$\lambda$ are contained in~$\mathbf{S}^1\cup\mathbf{R}$.
\item[(b)] For some positive integer~$k$,~$\lambda^k$ is the stretch factor of a 
pseudo-Anosov homeomorphism arising from Thurston's construction.
\item[(c)] For some positive integer~$k$,~$\lambda^k$ is the spectral radius of a 
bipartite Coxeter transformation of a bipartite Coxeter diagram with simple edges.

\end{enumerate}
\end{thm}

\noindent
This result is of optimal quality. Indeed, the smallest stretch factor of a pseudo-Anosov homeomorphism 
arising from Thurston's construction as well as the smallest spectral radius~$>1$ of a Coxeter 
transformation are both equal to Lehmer's number~$\lambda_L 
\approx 1.17628$ by work of Leininger~\cite{Leininger} and McMullen~\cite{McMullen}, respectively.
On the other hand, no such lower bound exists for bi-Perron numbers all of whose Galois conjugates are 
contained in~$\mathbf{S}^1\cup\mathbf{R}$. This is the content of the following proposition.

\begin{prop}
There exist bi-Perron numbers arbitrarily close to~$1$ and all of whose Galois conjugates are 
contained in~$\mathbf{S}^1\cup\mathbf{R}_{>0}$.
\end{prop} 

\noindent
The proof is short and we choose to give it here.

\begin{proof}
Choose any~$\varepsilon>0$. 
By Robinson's work on Chebyshev polynomials~\cite{Robinson}, there exist infinitely many algebraic integers 
that lie, together with all their Galois conjugates, in the interval~$[-2+\varepsilon, 2+2\varepsilon]$. On the 
other hand, by a result due to P\'olya described in Schur~\cite{Schur}, only finitely many algebraic integers 
lie, together with all their Galois conjugates, in the interval~$[-2+\varepsilon, 2]$. It follows that in the 
interval~$(2,2+2\varepsilon]$, there exist infinitely many Perron numbers all of whose Galois conjugates are 
contained in the interval~$[-2+\varepsilon, 2+2\varepsilon]$. Let~$p(t)$ be the minimal polynomial of such a 
Perron number and define the polynomial~$f(t) = t^{\mathrm{deg}(p)}p(t+t^{-1})$. Then every 
root~$x$ of~$f(t)$ is related to some root~$y$ of~$p(t)$ by~$x+x^{-1} = y$ and vice versa. 
In particular, all the roots of~$f(t)$ are contained in~$\mathbf{S}^1\cup\mathbf{R}_{>0}$. Furthermore, 
if~$2<y<2+2\varepsilon$, then~$1<x<1+\varepsilon+\sqrt{2\varepsilon+\varepsilon^2}$, assuming without loss of generality that~$x>x^{-1}$. 
Now, let~$x_0$ be the maximal real root of~$f(t)$. 
By construction, no other root of~$f(t)$ is as small as~$x_0^{-1}$ in modulus, so~$x_0$ is a bi-Perron number 
all of whose Galois conjugates are contained in~$\mathbf{S}^1\cup\mathbf{R}_{>0}$.  
Choosing~$\varepsilon$ arbitrarily small yields the desired result.
\end{proof}

\noindent
For the Galois conjugates of a bi-Perron number, we have the following result; the statement is 
different from the one of Theorem~\ref{bi-Perron_characterisation} in that 
we only have to use squares for the characterisation, and we only need Coxeter diagrams 
that are trees. 

\begin{thm}
\label{conjugates_characterisation}
For a Galois conjugate~$\lambda$ of a bi-Perron number, the following are equivalent.
\begin{enumerate}
\item[(a)] All Galois conjugates of~$\lambda$ are contained in~$\mathbf{S}^1\cup\mathbf{R}$.
\item[(b)] The number $\lambda^2$ is an eigenvalue of a Coxeter transformation associated 
with a tree.
\end{enumerate}
\end{thm}

\noindent
We note that the bi-Perron number in the statement might not be the spectral radius of 
the Coxeter transformation. Furthermore, we do not include a statement concerning stretch factors, 
since in the setting of Thurston's construction we cannot assure that~$\lambda$ is actually a Galois 
conjugate of a stretch factor, but only an eigenvalue of the action induced on the 
first homology of the surface by a pseudo-Anosov homeomorphism. 
\medskip

\noindent
Again, no result of the generality of Theorem~\ref{conjugates_characterisation} can be obtained 
without taking squares: by a result of A'Campo~\cite{A'Campo}, a Coxeter transformation associated 
with a tree has no negative real eigenvalue, except for possibly~$-1$. 
\medskip

\noindent
\textsc{Organisation}.
In the next section, we discuss some Galois-theoretic properties of Perron and bi-Perron numbers and prove a result related to trace fields of Perron numbers.
We also recall a result of the second author~\cite{Pankau} that is key for our purposes. 
In Section~\ref{Coxeter_pA_sec}, 
we describe the key input of Coxeter transformations and Thurston's construction of pseudo-Anosov 
homeomorphisms, and we prove Theorem~\ref{bi-Perron_characterisation} and 
Theorem~\ref{conjugates_characterisation}.
\medskip

\noindent
\textsc{Acknowledgements}. 
The first author would like to thank Ruth Kellerhals for inspiring discussions on the subject of bi-Perron numbers. 
We would also like to thank the anonymous referees for their helpful comments on previous versions of this article.

\section{Perron and bi-Perron numbers}
\label{Perron_sec}

\noindent
In this section, we prove a result about the trace fields of Perron and bi-Perron numbers.
A {Perron number}~$\lambda$ is a real algebraic integer~$>1$ all of whose Galois conjugates have modulus 
in the open interval~$(0, \lambda)$, except for~$\lambda$ itself.
The following statement is given in the proof of Lemma~8.2 of Strenner~\cite{Strenner}. 

\begin{lem}
\label{Perronpowers}
Let~$\lambda$ be a Perron number of degree~$l$, and let~$\lambda_1,\dots, \lambda_l$ be its Galois 
conjugates. Then for all positive integers~$k$,~$\lambda^k$ is also of degree~$l$ 
and~$\lambda_1^k,\dots,\lambda_l^k$ are its Galois conjugates. 
\end{lem}
\medskip

\noindent
The property of Perron numbers highlighted by this lemma is a key ingredient to proving the following proposition.

\begin{prop}
\label{Perrontraces}
Let $\lambda$ be a Perron number. Then, we have one of the following.\\

\begin{enumerate}
\item 
\label{firstprop}
If $[\mathbf{Q}(\lambda):\mathbf{Q}]$ is odd, then for all $k$ we have 
\[\mathbf{Q}(\lambda + \lambda^{-1}) = \mathbf{Q}(\lambda^k + \lambda^{-k}).\]
\item 
\label{secondprop}
If $[\mathbf{Q}(\lambda):\mathbf{Q}]$ is even, then for all $k$ we have
\[\mathbf{Q}(\lambda + \lambda^{-1}) = \mathbf{Q}(\lambda^{2k+1} + \lambda^{-(2k+1)})\]
and
\[\mathbf{Q}(\lambda^2 + \lambda^{-2}) = \mathbf{Q}(\lambda^{2k} + \lambda^{-2k})\]
with $\mathbf{Q}(\lambda + \lambda^{-1}) = \mathbf{Q}(\lambda^2 + \lambda^{-2})$ if and only if $-\lambda^{-1}$ is {not} a Galois conjugate of $\lambda$.
\end{enumerate}
 \end{prop}

\noindent
We note that it is possible for the stretch factor~$\lambda$ of an orientation-reversing pseudo-Anosov map to have~$-\lambda^{-1}$ as a Galois conjugate. 
The following example describes an instance of this phenomenon in the simplest case of the torus.

\begin{example}\emph{
The golden ratio~$\phi = \frac{1 + \sqrt{5}}{2}$ is a bi-Perron number with minimal polynomial~$t^2 - t - 1$. 
By definition,~$-\phi^{-1} = \frac{1 - \sqrt{5}}{2}$ is a Galois conjugate of~$\phi$, 
and we have that $\mathbf{Q}(\phi) = \mathbf{Q}(\phi + \phi^{-1}) = \mathbf{Q}(\sqrt{5})$ but $\mathbf{Q}(\phi^2 + \phi^{-2}) = \mathbf{Q}$.
While the golden ratio is not the stretch factor of an orientation-preserving Anosov map of the torus, it is the stretch factor of an orientation-reversing one: 
indeed, the spectral radius of the matrix~$\begin{pmatrix}  1 & 1 \\ 1 & 0\end{pmatrix}$ is the golden ratio. 
}
\end{example}

\noindent
We now prove Proposition~\ref{Perrontraces}, which will be important in the proof of Theorem~\ref{bi-Perron_characterisation} below. 

\begin{proof}[Proof of Proposition~\ref{Perrontraces}]
We start by noting that since~$\lambda$ is a Perron number, then Lemma~\ref{Perronpowers} tells us that~$[\mathbf{Q}(\lambda):\mathbf{Q}] = [\mathbf{Q}(\lambda^k):\mathbf{Q}]$ for all positive integers~$k$. This immediately implies that~$\mathbf{Q}(\lambda) = \mathbf{Q}(\lambda^k)$ for all~$k$, since the former is a field extension of the latter.
\medskip

\noindent
We now prove part~(\ref{firstprop}) of the proposition by assuming that~$[\mathbf{Q}(\lambda):\mathbf{Q}]$ is odd. Since~$\mathbf{Q}(\lambda^k)$ is a field extension of~$\mathbf{Q}(\lambda^k+\lambda^{-k})$, and since~$\lambda^k$ is a root of the polynomial~$t^2 -(\lambda^k+\lambda^{-k})t + 1$, then we must have that 
\begin{align}
\label{degree_rel}
\tag{$\ast$}
[\mathbf{Q}(\lambda^k):\mathbf{Q}(\lambda^k + \lambda^{-k})] = 1 \ \text{or} \ 2. 
\end{align}

\noindent
Now this degree cannot equal 2 because $\mathbf{Q}(\lambda) = \mathbf{Q}(\lambda^k)$, hence $[\mathbf{Q}(\lambda^k):\mathbf{Q}]$ is odd, so by the tower theorem for field extensions, none of the intermediate extensions can have even degree.  Hence, $\mathbf{Q}(\lambda)=\mathbf{Q}(\lambda^k) = \mathbf{Q}(\lambda^k + \lambda^{-k})$ for all positive integers~$k$. In particular, we have that~$\mathbf{Q}(\lambda + \lambda^{-1}) = \mathbf{Q}(\lambda^k + \lambda^{-k})$ for all $k$.
\medskip

\noindent
The proof of part~(\ref{secondprop}) of the proposition will be broken into two steps. The first step will be to prove that the equality $\mathbf{Q}(\lambda + \lambda^{-1}) = \mathbf{Q}(\lambda^2 + \lambda^{-2})$ holds if and only if $-\lambda^{-1}$ is not a Galois conjugate of $\lambda$. We do this by proving the contrapositive. The second step will be to prove the general equalities in the statement by considering the cases for when $-\lambda^{-1}$ is a Galois conjugate or not.
\medskip

\noindent
It is important to note that $-\lambda^{-1}$ can only be a Galois conjugate in the even degree case since if it is a conjugate, then for any other conjugate $\mu$, then so is~$-\mu^{-1}$. Hence the minimal polynomial has an even number of roots.
\medskip

\noindent
We start by assuming that $[\mathbf{Q}(\lambda): \mathbf{Q}]$ is even. Now, because of the inclusions~$\mathbf{Q}(\lambda^2 + \lambda^{-2}) \subseteq \mathbf{Q}(\lambda + \lambda^{-1}) \subseteq \mathbf{Q}(\lambda) = \mathbf{Q}(\lambda^2)$ then we have the following tower, with the possible degree of each extension listed:

\[\begin{array}{c}
		\mathbf{Q}(\lambda) = \mathbf{Q}(\lambda^2)\\
		\hspace{.4in}\Big\vert \text{ {\small $1$ or $2$} }\\
		\mathbf{Q}(\lambda + \lambda^{-1})\\
		\hspace{.4in}\Big\vert \text{ {\small $1$ or $2$} }\\
		\mathbf{Q}(\lambda^2 + \lambda^{-2})
	\end{array}\]
\medskip

\noindent
We start by assuming that $\mathbf{Q}(\lambda + \lambda^{-1}) \neq \mathbf{Q}(\lambda^2 + \lambda^{-2})$, and show that $-\lambda^{-1}$ must be a Galois conjugate of $\lambda$. Since these fields are not equal, then we immediately have that the top extension must be degree $1$ because we know from~(\ref{degree_rel}) that the degree of $\mathbf{Q}(\lambda^2)$ over $\mathbf{Q}(\lambda^2 + \lambda^{-2})$ is $1$ or $2$. Therefore, we have~$\mathbf{Q}(\lambda + \lambda^{-1}) = \mathbf{Q}(\lambda) = \mathbf{Q}(\lambda^2)$, and the tower collapses to

\[\begin{array}{c}
		\mathbf{Q}(\lambda + \lambda^{-1}) = \mathbf{Q}(\lambda) = \mathbf{Q}(\lambda^2)\\
		\hspace{.2in}\Big\vert \ 2\\
		\mathbf{Q}(\lambda^2 + \lambda^{-2}).
	\end{array}\]
\medskip

\noindent
Hence, $t^2 - (\lambda^2 + \lambda^{-2})t + 1$ is the minimal polynomial for $\lambda^2$ over $\mathbf{Q}(\lambda^2 + \lambda^{-2})$. Thus, we see that the non-identity automorphism $\phi\in\text{Gal}(\mathbf{Q}(\lambda^2)/\mathbf{Q}(\lambda^2 + \lambda^{-2}))$ maps $\lambda^2$ to  $\lambda^{-2}$. Hence, $[\phi(\lambda)]^2 = \lambda^{-2}$, and we get $\phi(\lambda) = \pm \lambda^{-1}$. Note, we are allowed to apply $\phi$ to $\lambda$ in this case since $\mathbf{Q}(\lambda) = \mathbf{Q}(\lambda^2)$. Now, it cannot be the case that $\phi(\lambda) = \lambda^{-1}$ because this would imply that $\mathbf{Q}(\lambda + \lambda^{-1}) = \mathbf{Q}(\lambda^2)$ is fixed by $\text{Gal}(\mathbf{Q}(\lambda^2)/\mathbf{Q}(\lambda^2 + \lambda^{-2}))$, which contradicts the definition of the Galois group. Therefore, $\phi(\lambda)=-\lambda^{-1}$, which implies that $-\lambda^{-1}$ is a Galois conjugate of $\lambda$.
\medskip

\noindent
Now, running the argument in reverse, if $-\lambda^{-1}$ is a Galois conjugate of $\lambda$ then there exists a $\mathbf{Q}$-automorphism $\phi$ of $\mathbf{Q}(\lambda)$ such that $\phi(\lambda) = -\lambda^{-1}$. This immediately implies that $\mathbf{Q}(\lambda^2 + \lambda^{-2})$ is fixed by $\phi$ but $\mathbf{Q}(\lambda + \lambda^{-1})$ is not, therefore, we have $\mathbf{Q}(\lambda + \lambda^{-1}) \neq \mathbf{Q}(\lambda^2 + \lambda^{-2})$.
\medskip

\noindent
We now generalize the argument and prove the equalities in statement~(\ref{secondprop}) of the proposition. Suppose that $-\lambda^{-1}$ is a Galois conjugate of~$\lambda$. Then, the automorphism 
$\phi$ that interchanges $\lambda$ and $-\lambda^{-1}$ fixes $\mathbf{Q}(\lambda^{2k} + \lambda^{-2k})$ for all positive integers~$k$. This implies that \[ [\mathbf{Q}(\lambda^{2k}): \mathbf{Q}(\lambda^{2k} + \lambda^{-2k})] = 2\] 
for all $k$. But $\mathbf{Q}(\lambda^{2k} + \lambda^{-2k})$ is a subfield of $\mathbf{Q}(\lambda^2 + \lambda^{-2})$ for all $k$ so \[\mathbf{Q}(\lambda^2 + \lambda^{-2}) = \mathbf\mathbf{Q}(\lambda^{2k} + \lambda^{-2k}) \] for all $k$. On the other hand $\mathbf{Q}(\lambda^{2k+1} + \lambda^{-(2k+1)})$ is not fixed for any $k$, 
hence \[ \mathbf{Q}(\lambda^{2k+1} + \lambda^{-(2k+1)}) = \mathbf{Q}(\lambda^{2k + 1}) = \mathbf{Q}(\lambda) = \mathbf{Q}(\lambda + \lambda^{-1}) \] for all $k$.
\medskip

\noindent
Now, in the case where $-\lambda^{-1}$ is not a conjugate of $\lambda$, above arguments immediately imply that $\mathbf{Q}(\lambda + \lambda^{-1}) = \mathbf{Q}(\lambda^2 + \lambda^{-2})$. Therefore, both of these fields must either equal $\mathbf{Q}(\lambda) = \mathbf{Q}(\lambda^k)$ for all positive integers~$k$, or for no $k$. Suppose that the following equality holds for all $k$:
\[\mathbf{Q}(\lambda + \lambda^{-1}) = \mathbf{Q}(\lambda^2 + \lambda^{-2}) = \mathbf{Q}(\lambda^k).\]
Then, it must be the case that $\mathbf{Q}(\lambda^{2k} + \lambda^{-2k}) = \mathbf{Q}\left(\lambda^{2k+1} + \lambda^{-(2k+1)}\right) = \mathbf{Q}(\lambda^k)$ for all $k$, because if equality fails for some $k$, then there would have to exist a~$\mathbf{Q}(\lambda^{n} + \lambda^{-n})$-automorphism $\phi$ (where $n=2k$ or $n=2k+1$)  that interchanges~$\lambda^n$ with~$\lambda^{-n}$. Thus, $\phi(\lambda) = \lambda^{-1}$, since it cannot equal $-\lambda^{-1}$ by assumption. Hence, $\mathbf{Q}(\lambda + \lambda^{-1}) = \mathbf{Q}(\lambda)$ is a fixed field, which is a contradiction. Therefore, for all $k$ we have \[\mathbf{Q}(\lambda^{2k} + \lambda^{-2k}) = \mathbf{Q}\left(\lambda^{2k+1} + \lambda^{-(2k+1)}\right) = \mathbf{Q}(\lambda^k).\]
Finally, if the fields $\mathbf{Q}(\lambda + \lambda^{-1}) = \mathbf{Q}(\lambda^2 + \lambda^{-2})$ do not equal $\mathbf{Q}(\lambda) = \mathbf{Q}(\lambda^k)$ for any $k$, then since $\mathbf{Q}(\lambda^k + \lambda^{-k})$ is a subfield of $\mathbf{Q}(\lambda + \lambda^{-1})$ for all $k$,~(\ref{degree_rel}) implies that $[\mathbf{Q}(\lambda^k):\mathbf{Q}(\lambda^k + \lambda^{-k})] = 2$ for all $k$, therefore \[\mathbf{Q}(\lambda + \lambda^{-1}) = \mathbf{Q}(\lambda^2 + \lambda^{-2}) = \mathbf{Q}(\lambda^k + \lambda^{-k})\] for all $k$.
\end{proof}

\noindent
The following result is the key to the construction of a geometric situation that corresponds to the 
power of a bi-Perron number. It was used by the second author to show that every Salem number has a 
power that is the stretch factor of a pseudo-Anosov homeomorphism arising from 
Thurston's construction~\cite{Pankau}. We now show that the proof works almost 
identically for bi-Perron numbers all of whose Galois conjugates are contained in~$\mathbf{S}^1\cup\mathbf{R}$. 
This extension is also presented in the second author's thesis~\cite{PankauThesis}.

\begin{prop}
\label{integer_matrix}
Let~$\lambda$ be a bi-Perron number all of whose Galois conjugates are contained 
in~$\mathbf{S}^1\cup\mathbf{R}$. Then there exists a positive integer~$k$ so that~$\lambda^k + \lambda^{-k}$ 
equals the spectral radius of a positive symmetric integer matrix. 
\end{prop}

\begin{proof}
We very closely follow the second author's proof in \cite{Pankau}. For the convenience of the reader, 
we summarise the key steps and mention where we have to pay attention because
our setting is slightly more general than in the original argument. 
\medskip

\noindent
Let~$\lambda$ be a bi-Perron number all of whose Galois conjugates lie 
in~$\mathbf{S}^1\cup\mathbf{R}$. Then~$\lambda+\lambda^{-1}$ is a totally real Perron number. 
Indeed,~$\mathbf{Q}(\lambda+\lambda^{-1})$ is a subfield of~$\mathbf{Q}(\lambda)$, so
every embedding of~$\mathbf{Q}(\lambda+\lambda^{-1})$ into~$\mathbf{C}$ is the 
restriction of an embedding of~$\mathbf{Q}(\lambda)$ into~$\mathbf{C}$. In particular, each Galois conjugate 
of~$\lambda+\lambda^{-1}$ is of the form~$\lambda_i + \lambda_i^{-1}$, where~$\lambda_i$ is a Galois conjugate of~$\lambda$. 
Since~$\lambda_i \in \mathbf{S}^1\cup\mathbf{R}$, it follows that~$\lambda_i + \lambda_i^{-1} \in \mathbf{R}$.
\medskip

\noindent
Let~$f(t)$ be the minimal polynomial of~$\lambda+\lambda^{-1}$, and denote by~$n$ its degree. 
Without loss of generality, we assume that~$-\lambda^{-1}$ is not a Galois conjugate of~$\lambda$. 
Indeed, if~$-\lambda^{-1}$ is a Galois conjugate of~$\lambda$, we can simply run the argument for~$\lambda^2$.
\smallskip
 
\noindent
\emph{Step 1.} By a result of Estes~\cite{Estes}, there exists a rational symmetric matrix~$Q$ of size~$(n+e)\times(n+e)$
with characteristic polynomial~$f(t)(t-1)^e$, where~$e$ equals~$1$ or~$2$. 
\medskip

\noindent
\emph{Step 2.} By conjugation with an element in~$\mathrm{O}(n+e,\mathbf{Q})$ and possibly a small perturbation,
we may assume that the eigenvector of the matrix~$Q$ for the eigenvalue~$\lambda+\lambda^{-1}$ is positive, compare with 
the discussion starting with Proposition~5.2 in \cite{Pankau}.
\medskip

\noindent
\emph{Step 3.} Define the matrix~\[ \mathcal{M} = \begin{pmatrix} Q & -I \\ I & 0 \end{pmatrix}.\]
We now describe the characteristic polynomial of~$\mathcal{M}$. In the proof of Proposition~5.3 in
\cite{Pankau}, it is shown that~$\mu$ is an eigenvalue for~$\mathcal{M}$ with 
eigenvector~$(\mathbf{v},\mu^{-1}\mathbf{v})^\top$ exactly if~$\mu + \mu^{-1}$ is an eigenvalue for~$Q$ with 
eigenvector~$\mathbf{v}$. Hence, the characteristic polynomial of~$\mathcal{M}$ equals~$t^nf(t+t^{-1})(t^2-t+1)^e$. 
We note the following discrepancy with Proposition~5.3 in \cite{Pankau}: if the characteristic 
polynomial~$g(t)$ of~$\lambda$ is not reciprocal, then the polynomial~$t^nf(t+t^{-1})$ 
equals~$g(t)g^*(t)$, where~$g^*(t)=t^ng(t^{-1})$. On the other hand, if~$g(t)$ is reciprocal, 
which is the case exactly if~$\lambda$ has a Galois conjugate on the unit circle 
(for example if~$\lambda$ is a Salem number), then~$t^nf(t+t^{-1})$ equals~$g(t)$. 
In any case, the characteristic polynomial of~$\mathcal{M}$ has integer coefficients and~$\det(\mathcal{M})=1$. 
\medskip

\noindent
\emph{Step 4.} By Proposition~5.4 in \cite{Pankau}, for any positive integer~$k$,~$\mathcal{M}^k+\mathcal{M}^{-k}$ is a block diagonal matrix with two blocks~$\mathcal{Q}_k$. 
Here,~$\mathcal{Q}_k$ is a rational symmetric matrix with characteristic polynomial~$f_k(t)(t-a)^e$, 
where~$f_k(t)$ is the minimal polynomial of~$\lambda^k+\lambda^{-k}$ 
and~$a$ is among the numbers~$-2,-1,1,2$. The proof does not depend on whether
the characteristic polynomial~$g(t)$ of~$\lambda$ is reciprocal or not. Also, by the discussion right above 
Proposition~5.5 in \cite{Pankau}, the eigenspaces of~$Q$ and~$\mathcal{Q}_k$ agree. In particular, 
the eigenvector~$\mathbf{v}$ for the eigenvalue~$\lambda^k+\lambda^{-k}$ of~$\mathcal{Q}_k$ is positive.
\medskip

\noindent
\emph{Step 5.} 
We now prove that the matrix~$\mathcal{Q}_k$ is positive for~$k$ large enough. 
We write~$\mathbf{e}_i = c_i\mathbf{v} + \mathbf{w}_i$ for every basis vector~$\mathbf{e}_i$, 
where~$\mathbf{w}_i$ is a fixed vector (independent of~$k$) in the orthogonal complement of~$\mathbf{v}$, and~$c_i>0$.
Since~$\mathbf{w}_i$ lies in the orthogonal complement to~$\mathbf{v}$, 
it is a linear combination of eigenvectors of~$\mathcal{Q}_k$ other than~$\mathbf{v}$. 
In particular, the modulus of every coefficient of~$\mathcal{Q}_k\mathbf{w}_i$ is bounded from above 
by~$|\lambda_2^k + \lambda_2^{-k}|\cdot ||\mathbf{w}_i||_{\infty}$, 
where~$\lambda_2+\lambda_2^{-1}$ is the second-largest root in modulus of~$f(t)$. 
Now, since~$\lambda$ is a bi-Perron number and~$-\lambda^{-1}$ is not among its Galois conjugates,
the ratio between~$\lambda^k+\lambda^{-k}$ and~$\lambda_2^k + \lambda_2^{-k}$ becomes arbitrarily large 
when~$k$ tends to infinity. Therefore,~\[\mathcal{Q}_k\mathbf{e}_i = c_i(\lambda^k+\lambda^{-k})\mathbf{v} + \mathcal{Q}_k\mathbf{w}_i\] becomes positive for large~$k$, 
since~$c_i>0$ and~$\mathbf{v}$ 
is a positive vector.
\medskip

\noindent
\emph{Step 6.} For large enough~$k$, the matrix~$\mathcal{M}^k$ has integer coefficients by 
Proposition~5.5 in \cite{Pankau}. Hence, also~$\mathcal{M}^{-k}$ has integer coefficients 
for large enough~$k$, since~$\det(\mathcal{M})=1$. In particular, also~$\mathcal{Q}_k$ has integer 
coefficients for large enough~$k$. This finishes the proof that for~$k$ large enough, the 
number~$\lambda^k+\lambda^{-k}$ equals the spectral radius 
of a positive symmetric integer matrix~$\mathcal{Q}_k$. 
\end{proof}

\section{The Coxeter transformations and Thurston's construction}
\label{Coxeter_pA_sec}

\subsection{The Coxeter transformation}

\noindent
Coxeter groups are abstract generalisations of reflection groups. They admit a presentation encoded in a 
graph with weighted edges, the so-called Coxeter diagram, and they are linear by Tits' representation. 
As we can single out the only input we need from the theory of Coxeter groups in Lemma~\ref{spectra_relation} 
below, we do not give the definitions and instead refer to Bourbaki's classic~\cite{Bourbaki}.  
\medskip

\noindent
In case the underlying graph of a Coxeter diagram is bipartite, there is a well-defined conjugacy class 
of matrices obtained via Tits' representation, the so-called bipartite Coxeter transformation, see, for example, 
McMullen~\cite{McMullen}. By a result of A'Campo, the spectrum of this matrix is contained 
in~$\mathbf{S}^1 \cup \mathbf{R}_{>0}$, and determines, for example, whether the group is finite~\cite{A'Campo}.  
All we need for our purposes is the following formula relating the spectra of the Coxeter adjacency matrix 
and the bipartite Coxeter transformation. We do not give the definition of the Coxeter 
adjacency matrix, but simply note that in our case of Coxeter diagrams with simple edges 
(which, in the language of Coxeter groups, means that every edge is of weight~3), the Coxeter adjacency 
matrix equals the ordinary adjacency matrix of the underlying abstract graph.

\begin{lem}
\label{spectra_relation}
Let~$\Omega$ be the adjacency matrix of a finite bipartite graph with simple edges, understood as a 
Coxeter diagram~$\Gamma$ with edge weights equal to~$3$. Then the eigenvalues~$\lambda_i$ 
of the bipartite Coxeter transformation associated with~$\Gamma$ are related to the 
eigenvalues~$\alpha_i$ of~$\Omega$ by \[ \alpha_i^2 -2= \lambda_i+\lambda_i^{-1} .\]
\end{lem}

\begin{proof}
This is exactly what is shown in the proof of Proposition~5.3 of McMullen's article~\cite{McMullen}.
\end{proof}

\subsection{Pseudo-Anosov stretch factors and Thurston's construction}
A homeomorphism~$\phi$ of a closed surface~$\Sigma$ is called pseudo-Anosov if there exists 
a pair of transverse, singular measured foliations of~$\Sigma$ so that~$\phi$ stretches one of them by a real 
number~$\lambda>1$ and shrinks the other by a factor~$1/\lambda$. The number~$\lambda$ is 
called the stretch factor of the pseudo-Anosov homeomorphism, and it is known to always be a 
bi-Perron number by a result of Fried~\cite{Fried85}. 
\medskip

\noindent
Thurston gave a construction of pseudo-Anosov homeomorphisms in terms of twists along 
multicurves~\cite{Thurston}. We do not review the whole construction, but instead summarise 
the main input we need from it in Lemma~\ref{Thurston_lemma} below.

\begin{lem}
\label{Thurston_lemma}
For a bi-Perron number~$\lambda>1$, the following are equivalent.
\begin{enumerate}
\item The number~$\lambda$ is the stretch factor of a pseudo-Anosov homeomorphism arising 
from Thurston's construction. 
\item The number~$\lambda$ is the spectral radius of a product of matrices 
\[ \begin{pmatrix} 1 & r \\ 0 & 1 \end{pmatrix}, \begin{pmatrix} 1 & 0 \\ -r & 1\end{pmatrix},\]
where~$r$ is the spectral radius of an adjacency matrix of a finite bipartite graph with simple edges.
\end{enumerate}
\end{lem}

\noindent
In order to give the proof, 
we first review some notions of surface topology. A multicurve~$\alpha$ in a surface~$\Sigma$ is a 
disjoint union of simple closed curves. Two multicurves are said to intersect minimally if any pair 
of components has the minimal number of intersection points among all representatives in their 
respective isotopy classes. Two multicurves fill a closed surface~$\Sigma$ if the complement of their union consists 
of discs. Furthermore, given two 
multicurves~$\alpha= \alpha_1\cup\cdots\cup \alpha_n$ and~$\alpha'=\alpha_{n+1}\cup\cdots\cup\alpha_{n+m}$,
we define their geometric intersection matrix to be the matrix of size~$(n+m)\times(n+m)$ whose~$ij$-th 
entry equals the number of intersection points of~$\alpha_i$ and~$\alpha_j$.  

\begin{proof}[Proof of Lemma~\ref{Thurston_lemma}]
Thurston's construction~\cite{Thurston} directly implies the statement for the 
case where~$r$ is the spectral radius of a geometric intersection matrix of multicurves~$\alpha$ 
and~$\alpha'$ that intersect minimally and fill a surface~$\Sigma$. It therefore suffices to show that 
the set of numbers that appear as spectral radii of such intersection matrices equals the set of numbers 
that appear as the spectral radii of adjacency matrices of finite bipartite graphs with simple edges. 
\medskip

\noindent
Given two multicurves~$\alpha$ and~$\alpha'$ that intersect minimally and fill a surface~$\Sigma$,
their geometric intersection matrix is, by definition, a symmetric nonnegative integer matrix.  In the proof of 
Proposition~2.1 of Hoffman~\cite{Hoffman}, it is shown that any spectral radius of such a matrix is also the 
spectral radius of an adjacency matrix of a finite bipartite graph with simple edges. This proves the first direction. 
\medskip

\noindent
Conversely, let~$A$ be the adjacency matrix of a finite bipartite graph with simple edges. 
Since the matrix~$A$ is conjugate to a diagonal block matrix with each block on the diagonal corresponding to a 
connected component of the graph, we can restrict to a block realising the spectral radius of~$A$. 
In other words, we assume without loss of generality that the graph is connected. 
\medskip

\noindent
Note that if the connected finite bipartite graph we consider has only a single vertex, then this implies~$r=0$. 
In particular, the only product of matrices we get in Lemma~\ref{Thurston_lemma} is the identity matrix, which has spectral radius~$1$.
This case is irrelevant, since we restrict ourselves to bi-Perron numbers~$\lambda > 1$.  
We can therefore assume that the connected finite bipartite graph has at least two vertices.
It is now straightforward to abstractly construct a closed surface~$\Sigma$ filled by two multicurves~$\alpha$ and~$\alpha'$ 
that have the matrix~$A$ as their geometric intersection matrix. 
In order to do so, take two collections of annuli~$K_i$ and~$K'_j$ that are in one-to-one correspondence with the vertices of the bipartite graph, 
respecting the bipartition. For each edge of the bipartite graph, locally identify the 
annuli~$K_i$ and~$K'_j$ corresponding to the endpoints of the edge along a common square whose boundary alternatingly belongs to the boundary of~$K_i$ and~$K'_j$. 
We glue such that the orientation of the annuli is respected and their core curves intersect once, see Figure~\ref{surfex} for an example 
of such a glueing corresponding to the complete bipartite graph~$K_{2,3}$ on two and three vertices.

\begin{figure}[h]
\begin{center}
\def\svgwidth{400pt}
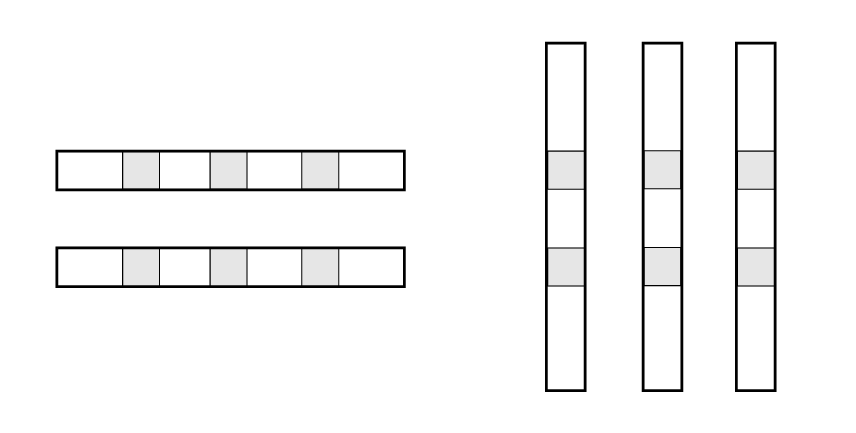
\caption{Two collections of annuli, horizontal and vertical, obtained by identifying the boundary of the rectangles as indicated by the letters~$a,b,c,d,e$. 
After identifying the squares labelled~$s_1,\dots,s_6$ pairwise by translations, the intersection matrix of the core curves of the annuli equals the adjacency 
matrix of the complete bipartite graph~$K_{2,3}$ on two and three vertices. In order to obtain the adjacency matrix of a subgraph, simply omit some of the identifications.}
\label{surfex}
\end{center}
\end{figure}

\noindent
So far, we have constructed a compact surface with boundary. 
To finish, glue a disc along each boundary component to obtain a closed surface~$\Sigma$. 
By construction, the core curves of the annuli~$K_i$ and~$K'_j$ define two multicurves~$\alpha$ and~$\alpha'$, 
respectively, filling~$\Sigma$ and with geometric intersection matrix~$A$.
Furthermore, the multicurves~$\alpha$ and~$\alpha'$ must intersect minimally, since simple closed 
curves with zero or one point of intersection always minimise the number of intersections within their respective isotopy classes.
\end{proof}

\subsection{Proof of Theorem~\ref{bi-Perron_characterisation}}
We prove the following implications: (a) implies~(c) implies~(b) implies~(a). \medskip

\noindent
\emph{(a) implies (c)}: Let~$\lambda$ be a bi-Perron number all of whose Galois conjugates are 
contained in~$\mathbf{S}^1\cup\mathbf{R}$. By Proposition~\ref{integer_matrix}, there exists a 
positive symmetric integer matrix~$M$ that has~$\lambda^k + \lambda^{-k}$ as its spectral radius, 
for some positive integer~$k$. In the proof of Proposition~2.1 of Hoffman~\cite{Hoffman}, 
it is shown that any number that is the spectral radius of a positive symmetric integer matrix is also 
the spectral radius of an adjacency matrix of a finite bipartite graph with simple edges. 
In particular,~$\lambda^k + \lambda^{-k}$ is the spectral radius of an adjacency matrix~$\Omega$ 
of a bipartite graph~$\Gamma$ with simple edges. By Lemma~\ref{spectra_relation}, the spectral 
radius~$x$ of the bipartite Coxeter transformation associated with~$\Gamma$ equals~$\lambda^{2k}$. 
Indeed, we have~$(\lambda^k+\lambda^{-k})^2 -2 = x + x^{-1}$, which 
yields~$\lambda^{2k} + \lambda^{-2k} = x + x^{-1}$ and hence~$x=\lambda^{2k}$, 
as~$x\mapsto x+x^{-1}$ is a strictly monotonic function on~$[1,\infty)$.
\medskip

\noindent
\emph{(c) implies (b)}: In the above implication, we have seen that~$\lambda^{2k}$ is the spectral 
radius of a bipartite Coxeter transformation associated with a bipartite Coxeter diagram with simple 
edges if and only if~$\lambda^k + \lambda^{-k}$ is the spectral radius of an adjacency matrix~$\Omega$ 
of a finite bipartite graph~$\Gamma$ with simple edges. We now use Lemma~\ref{Thurston_lemma} 
for the matrix product 
\[ \begin{pmatrix} 1 & \lambda^k + \lambda^{-k} \\ 0 & 1 \end{pmatrix} 
\begin{pmatrix} 1 & 0 \\ -(\lambda^k + \lambda^{-k}) & 1\end{pmatrix} 
= \begin{pmatrix} 1-(\lambda^k + \lambda^{-k})^2 & \lambda^k + \lambda^{-k} \\ -(\lambda^k + \lambda^{-k}) & 1\end{pmatrix},\]
the trace of which equals~$2-(\lambda^k+\lambda^{-k})^2$. In particular, the eigenvalues must satisfy the 
equation~$-t-t^{-1} = (\lambda^k+\lambda^{-k})^2 -2 = \lambda^{2k}+\lambda^{-2k}$. Hence, the eigenvalue 
with larger modulus is~$-\lambda^{2k}$ and so~$\lambda^{2k}$ is the spectral radius of the matrix product. 
By Lemma~\ref{Thurston_lemma}, the number~$\lambda^{2k}$ is the stretch factor of a pseudo-Anosov
homeomorphism arising from Thurston's construction.
\medskip

\noindent
\emph{(b) implies (a)}: Assume that~$\lambda^k$ is the stretch factor of a pseudo-Anosov homeomorphism 
arising from Thurston's construction. By a result of Hubert and Lanneau, the associated trace 
field~$\mathbf{Q}(\lambda^k+\lambda^{-k})$ is totally real~\cite{HuLa}. \medskip

\noindent
We first consider the case where~$-\lambda^{-1}$ is not a Galois conjugate of~$\lambda$.
From Proposition~\ref{Perrontraces}, 
we know that~$\mathbf{Q}(\lambda+\lambda^{-1})$ equals~$\mathbf{Q}(\lambda^k + \lambda^{-k})$ 
for all positive integers~$k$. Hence, if the field~$\mathbf{Q}(\lambda^k+\lambda^{-k})$ is totally real, 
then obviously so must be~$\mathbf{Q}(\lambda + \lambda^{-1})$, and all Galois conjugates 
of~$\lambda+\lambda^{-1}$ must be real. We note that all Galois conjugates of~$\lambda$ are roots of 
the polynomial~$t^{\mathrm{deg}(p)}p(t+t^{-1})$, where~$p(t)$ is the minimal polynomial of~$\lambda+\lambda^{-1}$. 
In particular, all Galois conjugates~$\lambda_i$ of~$\lambda$ must 
satisfy~$\lambda_i+\lambda_i^{-1}\in\mathbf{R}$ and so~$\lambda_i\in\mathbf{S}^1\cup\mathbf{R}$. \medskip

\noindent
In the case where~$-\lambda^{-1}$ is a Galois conjugate of~$\lambda$, Proposition~\ref{Perrontraces}
shows that one out of~$\mathbf{Q}(\lambda+\lambda^{-1})$ and~$\mathbf{Q}(\lambda^2+\lambda^{-2})$ 
equals~$\mathbf{Q}(\lambda^k + \lambda^{-k})$. In the former case, we are done by the above argument. 
In the latter case, the same argument gives that all Galois conjugates~$\lambda_i^2$ of~$\lambda^2$ 
are contained in~$\mathbf{S}^1\cup\mathbf{R}$. Hence, all Galois conjugates~$\lambda_i$ of~$\lambda$ 
are contained in~$\mathbf{S}^1\cup\mathbf{R} \cup i\mathbf{R}$. We are done by the observation 
that no Galois conjugate of~$\lambda$ can be totally imaginary. Indeed, assume~$\lambda_i$ is such 
a Galois conjugate. Then also~$\overline{\lambda_i}$ is a Galois conjugate of~$\lambda$, 
and we have~$\lambda_i^2 = \overline{\lambda_i^2}$. As an irreducible integer polynomial has no multiple zeroes, 
this implies~$\deg(\lambda^2)<\deg(\lambda)$, a contradiction by Lemma~\ref{Perronpowers}. 
This finishes the proof of Theorem~\ref{bi-Perron_characterisation}. 

\begin{remark}\emph{
Our proof strategy of cyclically showing (a) implies~(c) implies~(b) implies~(a) allows us to single out Lemma~\ref{spectra_relation}
as the only input needed on Coxeter transformations. We note that while we do so fairly implicitly, 
one can explicitly compare bipartite Coxeter transformations 
with the elements of Thurston's construction defined by a product of exactly two multitwists, 
see Section~8 of Leininger~\cite{Leininger}, thus providing a more conceptual proof of~(c) implies~(b). 
}\end{remark}

\subsection{Proof of Theorem~\ref{conjugates_characterisation}}
We prove (a) implies~(b) implies~(a). \medskip

\noindent
\emph{(a) implies (b)}:
Let~$\lambda$ be a Galois conjugate of a bi-Perron number all of whose Galois conjugates are 
contained in~$\mathbf{S}^1\cup\mathbf{R}$. Then~$\lambda+\lambda^{-1}$ is a totally real 
algebraic integer, so by a theorem of Salez~\cite{Salez},~$\lambda+\lambda^{-1}$ is an eigenvalue 
of an adjacency matrix~$\Omega$ of a finite tree~$\Gamma$. By Lemma~\ref{spectra_relation}, 
the eigenvalues~$\rho_i$ of the bipartite Coxeter transformation associated with~$\Gamma$ seen 
as a bipartite Coxeter diagram with simple edges are related to the eigenvalues~$\alpha_i$ of~$\Omega$ by
\[ \alpha_i^2 -2= \rho_i+\rho_i^{-1} .\]
By plugging in~$\lambda + \lambda^{-1}$ for~$\alpha_i$ we see that
\[\lambda^2 + \lambda^{-2} = \rho_i + \rho_i^{-1}.\]
Hence we have~$\rho_i = \lambda^2$, that is,~$\lambda^2$ is an eigenvalue of the Coxeter 
transformation associated with~$\Gamma$. \medskip

\noindent
\emph{(b) implies (a)}: This follows from the result that all the eigenvalues of the Coxeter transformation of 
a tree are contained in~$\mathbf{S}^1\cup\mathbf{R}_{>0}$, due to A'Campo~\cite{A'Campo}. In particular, 
we have that all Galois conjugates of~$\lambda^2$ are contained in~$\mathbf{S}^1\cup\mathbf{R}_{>0}$. 
Now, since~$\lambda$ is a Perron number, if~$\lambda_1,\dots,\lambda_l$ are the Galois conjugates 
of~$\lambda$, then~$\lambda_1^2,\dots,\lambda_l^2$ are the Galois conjugates of~$\lambda^2$ by 
Lemma~\ref{Perronpowers}. Hence, all Galois conjugates of~$\lambda$ lie in~$\mathbf{S}^1\cup\mathbf{R}$.


\begin{thebibliography}{99}
\bibitem{A'Campo} N.\ A'Campo: \emph{Sur les valeurs propres de la transformation de Coxeter}, Invent.\ Math.~\textbf{33} (1976), no.~1, 61--67.
\bibitem{Bourbaki} N.\ Bourbaki: \emph{Groupes et alg\`ebres de Lie}, Chapitres IV, V, VI, Hermann, 1968, Paris.
\bibitem{Estes} D.\ R.\ Estes: \emph{Eigenvalues of symmetric integer matrices}, J.\ Number Theory~\textbf{42} (1992), no.~3, 292--296.
\bibitem{Fried85} D.\ Fried: \emph{Growth rate of surface homeomorphisms and flow equivalence}, Ergodic Theory Dynam.\ Systems~\textbf{5} (1985), no.~4, 539--563.
\bibitem{Hoffman} A.\ J.\ Hoffman: \emph{On limit points of spectral radii of non-negative symmetric integral matrices}, Graph Theory Appl., Proc.\ Conf.\ Western Michigan Univ.\ 1972, Lect.\ Notes Math.~\textbf{303} (1972), 165--172.
\bibitem{HuLa} P.\ Hubert, E.\ Lanneau: \emph{Veech groups without parabolic elements}, Duke Math.\ J.~\textbf{133} (2006), no.~2, 335--346.
\bibitem{Leininger} C.\ J.\ Leininger: \emph{On groups generated by two positive multi-twists: Teichm\"uller curves and Lehmer's number},  Geom.\ Topol.~\textbf{8} (2004), 1301--1359. 
\bibitem{McMullen} C.\ T.\ McMullen: \emph{Coxeter groups, Salem numbers and the Hilbert metric}, Publ.\ Math.\ Inst.\ Hautes \'Etudes Sci.~\textbf{95} (2002), 151--183.
\bibitem{Pankau} J.\ Pankau: \emph{Salem number stretch factors and totally real number fields arising from Thurston's construction}, Geom.\ Topol.~\textbf{24}, no.~4, 1695--1716.
\bibitem{PankauThesis} J.\ Pankau: \emph{On stretch factors of pseudo-Anosov maps}, PhD thesis, University of California Santa Barbara, 2018. 
\bibitem{Robinson} R.\ M.\ Robinson: \emph{Intervals containing infinitely many sets of conjugate algebraic integers}, Studies in Mathematical Analysis and Related Topics: Essays in Honor of George P\'olya, Stanford 1962, pp.\ 305--315.
\bibitem{Salez} J.\ Salez: \emph{Every totally real algebraic integer is a tree eigenvalue}, J.\ Comb.\ Theory, Ser.~B~\textbf{111} (2015), 249--256.
\bibitem{Schur} I.\ Schur: \emph{\"Uber die Verteilung der Wurzeln bei gewissen algebraischen Gleichungen mit ganzzahligen Koeffizienten}, Math.\ Z.~\textbf{1} (1918), no.~4, 377--402.
\bibitem{Strenner} B.\ Strenner: \emph{Algebraic degrees of pseudo-Anosov stretch factors}, Geom.\ Funct.\ Anal.~\textbf{27} (2017), no.~6, 1497--1539.
\bibitem{Thurston} W.\ Thurston: \emph{On the geometry and dynamics of diffeomorphisms of surfaces}, Bull.\ Am.\ Math.\ Soc.~\textbf{19} (1988), no.~2, 417--431.
\end{thebibliography}
\end{document}